\documentclass{article}
\usepackage{graphicx} 

\usepackage[margin=1in]{geometry}
\usepackage[absolute]{textpos}
\usepackage{amsmath,graphicx,float,amsfonts,amssymb,amsthm,amsmath,multicol,multirow,siunitx,setspace,tikz,xcolor,tcolorbox,textcomp,hyperref,pgfplots,booktabs,array,bigints,cancel,mathtools, import, CJKutf8, setspace, yfonts, mathrsfs,ulem,pgffor,vwcol}
\newcommand{\QED}{\textbf{\underline{QED}}}

\providecommand{\divtwo}[1]{$\left[ \frac{#1}{2} \right]$}
\providecommand{\divfour}[1]{$\left[ \frac{#1}{4} \right]$}
\providecommand{\powtwo}[1]{$\pi_{2}(#1)$}
\providecommand{\divpowtwo}[1]{$\left[ \frac{#1}{n} \right]$}

\usepackage{enumitem}

\title{\textbf{The Equations $\mathbf{2^n \pm 2^m \pm 1 = x^2}$ in the Arithmetic of the Even and Odd}}
\author{Matt Wicks}
\date{}
\newtheorem{theorem}{Theorem}
\newtheorem{lemma}{Lemma}
\begin{document}

\maketitle

\begin{abstract}
    The title equations were originally solved by making use of certain results on hypergeometric functions. Aside from these results, the classifications of the solutions uses very elementary arithmetic. The goal of this is to show that these solutions hold in a weak fragment of arithmetic; one strong enough to express the notions of even and odd that has been extended to make use of the results on hypergeometric functions. 
\end{abstract}

\section{Introduction}
Whenever an arithmetical statement that is expressible in elementary terms is proven, it is natural to ask which assumptions were necessary to carry out the proof of this statement. Usually, the axioms of Peano Arithmetic ($PA$) are enough. For the more logically inclined it can often be proved in $I\Delta_0+\exp$. Even then, both of these axiom systems are quite expansive, and the compactness theorem of first order logic guarantees that a proof of such a statement should only require finitely many of these axioms. Still, we want to ask for even more; we want to know the real reason that the statement is true. To this end, we seek a finite set of axioms that is strong enough to prove the statement, that also consists of assumptions which are number-theoretically meaningful and not just the minimal set of axioms that compactness theorem guarantees. In particular, in this paper we ask this about four Diophantine equations that involve powers of two. These equations ask for which $n,m$, $2^n \pm 2^m \pm 1$ is a perfect square. In order to express these equations we need some axioms that encapsulate the idea of a power of two, the minimal theory that allows these equations to be expressible has already been studied in \cite{schacht}, where it is referred to as the arithmetic of the even and the odd, or $\mathcal{AOE}$ for short. It's language does not contain any reference to exponentiation, instead it defines an element as a power of 2 if all of it's divisors, other than 1, are even. We will need to extend this system to include some axioms which are more specific to the situation for these equations involving particular inequalities and some additional assumptions that ensure that the powers of two behave well.\\

\section{Extending the Arithmetic of the Even and Odd}
This axiom system is itself an extension of the theory that is known as $PA-$\cite{kaye}, the first seven of these axioms establish that addition and multiplication are associative, commutative, satisfy the distributive law and that $0$ and $1$ serve as the identities for $+$ and $\cdot$ respectively. In all of the following axioms, universal quantifiers will be omitted.

\begin{itemize}
    \item [\textbf{A1.}] $x+y = y+x$
    \item [\textbf{A2.}] $(x + y) + z = x + (y + z)$
    \item [\textbf{A3.}] $x\cdot y = y \cdot x$
    \item [\textbf{A4.}] $x \cdot (y \cdot z) = (x \cdot y) \cdot z$
    \item [\textbf{A5.}] $x \cdot (y+z) = x \cdot y + x\cdot z$
    \item [\textbf{A6.}] $x + 0 = 0 + x = x$
    \item [\textbf{A7.}] $1 \cdot x = x \cdot 1 = x$\\
\end{itemize}
    
    We will also need an order that the operations respect and an operation $-$ to allow for subtraction.

\begin{itemize}[resume]
    \item [\textbf{A8.}] $x < y$ $\wedge$ $y < z \Rightarrow x < z$
    \item [\textbf{A9.}] $-x < x$
    \item [\textbf{A10.}] $x < y$ $\vee$ $x = y$ $\vee$ $y < x$
    \item [\textbf{A11.}] $x<y \implies x + z < y + z$
    \item [\textbf{A12.}] $0<z$ $\wedge$ $x<y \implies x \cdot z < y \cdot z$
    \item [\textbf{A13.}] $x<y \implies \exists z, x+z =y$\\
\end{itemize}
    The order should also be discrete, so that all elements of any model differ by at least $1$ and $0$ is the least element.

\begin{itemize}[resume]
    \item [\textbf{A14.}] $0 < 1$ $\wedge$ $x > 0 \implies x \geq 1$
    \item [\textbf{A15.}] $x \geq 0$\\
\end{itemize}

To classify the solutions to these equations we will need more than just $PA-$. On it's own $PA-$ is quite a weak theory; for instance it cannot prove the equivalence of primes and irreducibles. It is also unable to express what it means for an element to be a power of 2, something we will obviously require. We will also need the ability to divide not only by 2 but by powers of 2. For this we take the axiom system introduced in \cite{schacht}, which has a language of 0, 1, +, $\cdot$, $<$, \divtwo{\cdot}, and extend it to add what we will need. The unary function \divtwo{x} allows us to divide by two and satisfies the following axiom.

\begin{itemize}
    \item [\textbf{A16.}] \divtwo{2x} $= x$
\end{itemize}

We also make use of two unary symbols introduced in \cite{schacht}, $\omega$ and $\tau$ where we can always write\\
$n = \tau(n) \cdot \omega(n)$ with $\tau(n)$ being a power of two and $\omega(n)$ odd. This allows to say when a number is a power of 2, that is when $n = \tau(n)$. When $n$ is a power of two we use the unary predicate \powtwo{n} to stand for "n is a power of 2". Powers of two should satisfy the same properties as they do in the natural numbers, that they have only even divisors, that any number can be written as a power of two times an odd part, and that the difference of two powers of two is always divisible by a power of two.

\begin{itemize}
    \item [\textbf{A17.}] \powtwo{n} $\wedge$ $a \cdot b = n$ $\wedge$ $a > 1 \implies a = 2$\divtwo{a}
    \item [\textbf{A18.}] $0 < n \implies n = \tau(n) \cdot \omega(n)$ $\wedge$ \powtwo{\tau(n)} $\wedge$ $\omega(n) = 2$ \divtwo{\omega(n)} $+1$
    \item [\textbf{A19.}] $n < m$ $\wedge$ \powtwo{n} $\wedge$ \powtwo{m} $\implies \tau(m-n) = n$
\end{itemize}

The theory axiomatized by $A1 - A19$ is what is denoted by $\mathcal{AOE}$ in\cite{schacht}. We will need to extend this further; as a start we will need to take the results from~\cite{beukers} used in \cite{szalay} to classify the solutions to the title equations. However, in their original form they are not expressible in the language of $\mathcal{AOE}$, so we have to translate them into suitable expressions in terms of integers.

There are three major results used from\cite{beukers} in order to produce the solutions in\cite{szalay}, all of them are derived from results on hypergeometric functions, which are far outside the scope of the theory $\mathcal{AOE}$. The first is an estimate on the quotient of an integer by an odd power of two. We will say $p$ is an odd power of two if $p = 2\cdot q^2$ where \powtwo{q}. When written as an expression in integers, in the language of $\mathcal{AOE}$, theorem 1 of\cite{beukers} reads as:

\begin{itemize}
    \item [\textbf{B1.}] $ p = 2 \cdot q^2 \wedge$ \powtwo{q} $\implies \\ 2^{862}[2^{5}(11513\cdot q^{26}x^{10}) + 2^6(1965\cdot q^{24}\cdot x^{12}) + 2^{4}(1215\cdot q^{22}\cdot x^{14}) + 2^{4}(75\cdot q^{20}x^{16}) + 25\cdot q^{18}\cdot x^{88}] < 2^{860}[2^{10}\cdot q^{36} + 2^{10}\cdot 45\cdot q^{34}\cdot x^{2} + 2^{8}(2445\cdot q^{32}\cdot x^{4}) + 2^{10}\cdot 2415\cdot q^{30}\cdot x^{6} + 2^{7}(31545\cdot q^{28}\cdot x^{8}) + 2^{7}(23063\cdot q^{26}\cdot x^{10}) + 2^{5}(31545\cdot q^{24}\cdot x^{12}) + 2^{6}(2415\cdot q^{22}\cdot x^{14}) + 2^{2}(2445\cdot q^{20}\cdot x^{16}) + 2^{2}(45\cdot q^{18}\cdot x^{18}) + q^{16}x^{20}] - 2^{439}[2^{5}\cdot q^{18} + 2^{4}(45\cdot q^{16}\cdot x^{2}) + 2^{4}(105\cdot q^{14}\cdot x^{4}) + 2^{3}(45\cdot q^{10}\cdot x^{8}) + q^{8}\cdot x^{10}] + 1$  
\end{itemize}

The second result is the following restriction on solutions to a certain family of equations.
\begin{itemize}
    \item [\textbf{B2A.}] $D \neq 0$ $\wedge$ $D < 2^{96}$ $\wedge$ $\pi_{2}(p) \wedge p = x^{2} + D \implies p < 2^{210}$ 
\end{itemize}

\begin{itemize}
    \item [\textbf{B2B.}] $D \neq 0$ $\wedge$ $D < 2^{96}$ $\wedge$ $\pi_{2}(p) \wedge p + D = x^{2} \implies p < 2^{210}$ 
\end{itemize}

The third is a complete classification of when a certain family of equations has multiple solutions.
\begin{itemize}
    \item [\textbf{B3A.}] $\pi_{2}(n) \wedge \pi_{2}(m)\wedge n \neq m \wedge n - D = x^2 \wedge m - D = y^2 \implies [D = 7 \vee D = 23 \vee (D = k - 1 \wedge \pi_{2}(k) \wedge k \geq  16)]$ 
\end{itemize}
    
\begin{itemize}
    \item [\textbf{B3B.}] $\text{\powtwo{n}} \wedge n - 7 = x^2 \implies [(n = 8) \vee (n = 16) \vee (n = 32) \vee (n = 128) \vee (n = 32678)]$
\end{itemize}

Furthermore, we will need to be able to divide by arbitrary powers of two, if $n$ is a power of two we introduce the unary symbol \divpowtwo{\cdot} that satisfies the following axioms.

\begin{itemize}
    \item [\textbf{A20.}] $\text{\powtwo{n}} \implies$\divpowtwo{nx} $= x$
\end{itemize}

\begin{itemize}
    \item [\textbf{A21.}] $\pi_2(n)\wedge \pi_2(m)\implies [\frac{m}{n}]\cdot n = m$
\end{itemize}

For brevity, we will refer to the axioms $B2A \text{ and } B2B$ as simply $B2$ and similarly we will refer to the axioms $B3A \text{ and } B3B$ as $B3$. The theory axiomatized by $A1 - A21$ and $B1 - B3$ is the theory we will work in. For short, we denote it $\mathcal{AOE + B}$.

\section{Translating the Results}
All of these results follow the same arguments that were used in \cite{beukers}, just translated into the language of $\mathcal{AOE + B}$.
 \begin{lemma}
     If \powtwo{p},  then the solutions to the pair of  inequalities $0<|p - x^2| < 4$ are given by $(p, x) = (2,1), (4,1), (8,3), (4, 2)$.
 \end{lemma}

 \begin{proof}
     Suppose $(p,x)$ is a solution to the inequality, then we have that $p-x^2 = k$ where $|k| < 4$. By $B2$, $p < 2^{210}$. Performing a computer check of all possible values of $p$ gives the stated solutions. \QED
 \end{proof}
 
 \begin{lemma}
     If \powtwo{p} and $x>1$ and $y>1$ satisfy $y^2 -1 = p^{2}(x^{2}-1),$ then $x = \text{\divtwo{p}}$ and $y = \text{\divtwo{p^2}} - 1$.
 \end{lemma}

 \begin{proof}
     First, $p > 2$ since if we had a solution $(x,y)$ with $p = 2$ this would give that $y^2 - 1 = 4(x^2 - 1)$ or that $y^2 -4x^2 = -3$ which has no solutions with $x> 1$ and $y > 1$. Therefore, we may assume that $p \geq4$.\\ Now, $y^2$ is odd and thus $y$ is odd and we have that \divtwo{y-1}$\cdot$\divtwo{y+1} $= \text{\divtwo{p}}^{2}(x^{2}-1)$. Now, \divtwo{y-1} and \divtwo{y+1} are relatively prime and so $\text{\divtwo{p}}^{2}$ divides only one of them. Using this we can write $y = \text{\divtwo{p^{2}}}k \pm 1$. Because $x \text{ and } y$ are assumed to be solutions to the equation in the statement of the lemma, we must have that $y < px$ and consequently $\text{\divtwo{p}}k \leq x$. Now if $y = \text{\divtwo{p^{2}}}k + 1$ we see that $\text{\divfour{p^{4}}}k^{2} + p^{2}k = p^{2}(x^2 - 1)$. Dividing by $p^{2}$, which $A20$ allows us to do, shows that $\text{\divfour{p^{2}}}k^{2} + k = x^{2} - 1$. Repeating the same process beginning with $y = \text{\divtwo{p^{2}}}k - 1$ gives $\text{\divfour{p^{2}}}k^{2} - k = x^{2} - 1$. Taking $k = 1$ in second case gives the desired solutions, so we must show that this is the only possibility. In the first case, using that $\text{\divtwo{p}} < x$, squaring this inequality gives: 
     \begin{equation*}
        \text{\divtwo{p}}^{2} < \text{\divtwo{p}}^{2} + k + 1 = x^2 < (\text{\divtwo{p}} + 1)^{2}    
     \end{equation*}
     Which cannot happen in a discretely ordered ring. Now, in the second case if $k > 1$ then we have the following: 
    \begin{equation*}
        x^{2} = \text{\divfour{p^{2}}}k^{2} - (k - 1) < \text{\divfour{p^{2}}}k^{2} \leq x^{2}
    \end{equation*}
    Which is also a contradiction, therefore the only possibility is that $k = 1$ giving the desired solutions. \QED
 \end{proof}

 \vspace{10pt}

 With these two lemmas proved in the language of $\mathcal{AOE + B}$, we now have the tools required to prove Theorem 2 and Theorem 3 from \cite{szalay}. We prove these now and leave the discussion of Theorem 1 from \cite{szalay} until after.

 \begin{theorem}
     If \powtwo{n} and \powtwo{m} satisfy $n - m + 1 = x^2$ then $(n, m , x) \in \{ (t^{2}, 2t, t - 1): \text{\powtwo{t}} \wedge t \geq 4 \}$ or $(n, m, x) \in \{ (t, t, 1): \text{\powtwo{t}} \wedge t \geq 1 \}$ or $(n,m,x) \in \{ (32, 8, 5), (128, 8, 11), (32768, 8, 181) \}$ 
 \end{theorem}

 \begin{proof}
     First, suppose that \powtwo{n} and \powtwo{m} satisfy $n - m + 1 = x^2$ with $n \geq m \geq 16$. Taking $D_{2} = m - 1$, this equation becomes $n - D_{2} = x^{2}$. By $B3B$, this has solutions of the form $(n, x) = (m, 1) \text{ or } ([\frac{m}{2}]^{2}, \text{\divtwo{m}} - 1)$. In the case that $(n, x) = (m , 1)$ the equation reduces to $x^{2} = 1$ therefore $x = 1$ yielding the set of solutions of the form $\{ (t, t, 1): \text{\powtwo{t}} \wedge t \geq 1 \}$. In the case where $(n, x) = ([\frac{m}{2}]^{2}, \text{\divtwo{m}} - 1)$ these are exactly the solutions given in $\{ (t^{2}, 2t, t - 1): \text{\powtwo{t}} \wedge t \geq 4 \}$.
    Now, if $m = 8$ the equation has 5 solutions given in $B3B$. The solutions $(8,8,1)$ and $(16,8,3)$ already appear in the set of solutions above, the rest of the extraneous solutions are given in $B3B$. If $m = 1$ or $m = 2$ the only solutions are $(1,1,1)$ and $(2,2,1)$ which are already present in the sets given above. Finally, there can be no solutions if $n<m$ so all the solutions have been classified. \QED
\end{proof}

\begin{theorem}
    If \powtwo{n} and \powtwo{m} satisfy $n+m-1 = x^2$ then $(n,m,x) = (8,2,3)$. Further, if \powtwo{n} and \powtwo{m} satisfy $n-m-1 = x^2$ then $(n,m,x) = (4,2,1)$.
\end{theorem}

\begin{proof}
    First, let \powtwo{n} and \powtwo{m} satisfy $n+m-1 = x^2$ where $n \geq m > 0$ and additionally suppose $m \geq 4$. Then $n \geq 4$, which would imply that -1 is a quadratic residue modulo 4, which is false even in a theory as weak as $\mathcal{AOE + B}$. To see this, note that any element $y$ is of the form $y = 2k$ or $y = 2k + 1$, squaring gives that $y^2 = 4k^2$ or $y^2 = 4k^2 + 4k + 1$, thus the residues modulo 4 are still 0 and 1. With this in mind, it must be the case that $m < 4$ and since $m > 0$ this forces $m = 2$. The equation then becomes $n + 1 = x^2$ appealing to Lemma 1 and checking all possible solutions gives that $n = 8$ and $x = 3$ as desired. Now, let $n$ and $m$ be renewed and satisfy the equation $n - m - 1 = x^2$. Then, $n > m$ and again $m < 4$ otherwise -1 would be a quadratic residue modulo 4. Once again, $m = 2$ and the equation becomes $n - 2 = x^2$, appealing again to Lemma 1 and checking the possible solutions gives $(n, m , x ) = (4, 2, 1)$. \QED
\end{proof}

\vspace{10pt}

All that is left is to translate the first theorem from\cite{szalay}, this theorem requires more than the previous two but it can be translated in the same way that the others have been.

\begin{theorem}
    If \powtwo{n}, \powtwo{m} satisfy $n + m + 1 = x^2$ then $(n,m,x) \in \{ (t^2, 2t, t+1): \text{\powtwo{t}} \wedge t \geq 1 \} \text{ or } (n,m,x) \in \{(32, 16, 7), (512, 16, 23) \}$ 
\end{theorem}

The difficulty in proving this theorem comes from the fact that the strategy in the proof given in \cite{szalay} is to transform the equation into a form where $B1$ is applicable. Some of the lemmas used in this transformation are not provable within $\mathcal{AOE + B}$ and must be taken as axioms. Once these extra axioms are taken, it is possible, as we have done with the previous two theorems, to translate the proof given in \cite{szalay} into the language of $\mathcal{AOE+B}$.

\bibliographystyle{plain}
\bibliography{Powersoftwo}
\end{document}